\documentclass[11pt]{amsart}
\usepackage{amssymb}
\usepackage{amsthm,amsfonts}

\usepackage{amsfonts}
\usepackage{amsmath}
\usepackage[abbrev]{amsrefs}
\usepackage{enumerate}

\newtheorem{Theorem}{Theorem}[section]

\newtheorem{Lemma}[Theorem]{Lemma}
\newtheorem{Corollary}[Theorem]{Corollary}

\newcommand\supp{\mathop{\rm supp}}
\newcommand\sign{\mathop{\rm sign}}
\theoremstyle{definition}

\begin{document}

\title{$M$-ideal properties in Orlicz-Lorentz spaces}
\keywords{$M$-ideals, Orlicz-Lorentz spaces, dual norm}
\subjclass[2010]{46B20, 46E30, 47B38}

\author{Anna Kami\'nska}
\address{Department of Mathematical Sciences,
The University of Memphis, TN 38152-3240}
\email{kaminska@memphis.edu}

\author{Han Ju Lee}
\address{Department of Mathematical Education,
Dongguk University, Seoul, 100-715,  Republic of Korea}
\email{hanjulee@dongguk.edu}

\author{Hyung-Joon Tag}
\address{Department of Mathematical Sciences,
The University of Memphis, TN 38152-3240}
\email{htag@memphis.edu}
%\thanks{}

\date{\today}

\begin{abstract}
We provide  explicit formulas for the norm of  bounded linear functionals on Orlicz-Lorentz function spaces $\Lambda_{\varphi,w}$ equipped with two standard Luxemburg and Orlicz norms. Any bounded linear functional is a sum of regular and singular functionals, and we show that the norm of a singular  functional is the same regardless of the norm in the space, while  the formulas of the norm of  general functionals  are different for the Luxemburg and Orlicz norm. The relationship between equivalent definitions of the modular $P_{\varphi,w}$ generating the dual space to Orlicz-Lorentz space is discussed in order to compute the norm of a bounded linear functional on $\Lambda_{\varphi,w}$ equipped with Orlicz norm. As a consequence, we show that the order-continuous subspace of Orlicz-Lorentz space equipped with the Luxemburg norm is an $M$-ideal in $\Lambda_{\varphi,w}$, while this is not true for the space with the Orlicz norm  when $\varphi$ is an Orlicz $N$-function not satisfying the appropriate $\Delta_2$ condition. The analogous results on Orlicz-Lorentz sequence spaces are given.
\end{abstract}

\maketitle

\section{Introduction}

A closed subspace $Y$ of a Banach space $X$ is an $M$-ideal of $X$ if $Y^{\perp}$ is the range of the bounded projection $P:X^* \rightarrow X^*$ which satisfies
\[
\|x^*\| = \|Px^*\| + \|x^* - Px^*\| \ \ \ \text{for all}\ \ x^* \in X^*.
\]
 If $Y$ is an $M$-ideal in $X$, then each $y^* \in Y^*$ has a unique norm-preserving extension to $x^* \in X^*$\cite{HWW}.
It is well known that $c_0$ is an $M$-ideal in $l^\infty$. The $M$-ideal properties in Marcinkiewicz spaces have been studied in \cite{KL}. It was shown there that the subspace of order-continuous elements in $L^1 + L^{\infty}$ equipped with the standard norm is not an $M$-ideal, while there exists an equivalent norm such that this subspace is an $M$-ideal.  For Orlicz spaces $L_\varphi$ it is well known that the order-continuous subspace of $L_{\varphi}$ is an $M$-ideal if the space is equipped with the Luxemburg norm \cite{A, N}, while this is not true if the space is equipped with the Orlicz norm and if $\varphi$ does not satisfy the appropriate $\Delta_2$ conditions \cite{CH}. For more details of general $M$-ideal theory and their applications, we refer to \cite{HWW}.

In this article, we investigate Orlicz-Lorentz function and sequence spaces. While we obtain analogous results as in Orlicz spaces, the techniques are different and the calculations are more involved since there is  necessity to deal with decreasing rearrangements and level functions, and the K\"othe associate spaces to Orlicz-Lorentz spaces are not of the same sort  as in the case of Orlicz spaces. The exact isometric dual norm for regular functionals in Orlicz-Lorentz spaces  has been recently found in \cite{KLR} and it is expressed in terms of the Hardy-Littlewood order and the level functions. This paper completes the topic of characterization of the dual spaces by providing   exact formulas of dual norms to Orlicz-Lorentz spaces equipped with two standard Luxemburg and Orlicz norms.

Denote by $L^0 = L^0(I)$ the set of all Lebesgue measurable functions $f: I= [0, \gamma) \to \mathbb{R}$, where $0 < \gamma \leq \infty$. If $I=\mathbb{N}$ then $\ell^0 = L^0(\mathbb{N})$ denotes the collection of all real valued sequences $x = (x(i))$.  The interval $I=[0,\gamma)$ is equipped with the Lebesgue measure $m$, and the space $\ell^0 = L^0(\mathbb{N})$ with the counting measure $| \cdot |$. A Banach space $(X, \| \cdot \|)$ over $I$ is said to be a Banach function lattice if $X \subset L^0(I)$ and whenever $0 \leq x \leq y$,  $x \in L^0(I)$, $y \in X$, then $x \in X$ and $0 \leq \|x\| \leq \|y\|$. If $I=[0,\gamma)$ then $X$ is called a Banach function space, while if $I=\mathbb{N}$ then $X$ is called a Banach sequence space.

We say that a Banach function lattice $(X, \|\cdot\|)$ has the Fatou property provided that for every sequence $(x_n) \subset X$, if $x_n \uparrow x$ a.e. for $x \in L^0$  and $\sup_n \|x_n\| < \infty$, then $x \in X$ and $\|x_n\| \uparrow \|x\|$.  An element $x \in X$ is order-continuous if for any $0 \leq x_n \leq |x|$, if $x_n \downarrow 0$ a.e., then $\|x_n\| \downarrow 0$. The set of all order-continuous elements in $X$ is a closed subspace of $X$ and is denoted by $X_a$. We also define a subspace $X_b$ which is
the closure of the set of all simple functions with supports of finite measure. In general, $X_a \subset X_b$ \cite{BS}.

The K\"othe associate space of $X$, denoted by $X'$, is a subset of  $L^0(I)$, where $I = [0, \gamma)$, $0 < \gamma \leq \infty$, or $I = \mathbb{N}$ consisting of all $y \in X'$ satisfying $\|y\|_{X'}=\sup\{\int_I |xy|: \|x\|_X \leq 1\} < \infty$.  The space $X'$ equipped with the norm $\|\cdot\|_{X'}$ is a Banach function lattice. It is well known that $X$ has the Fatou property if and only if $X=X''$ \cite{Z}.

We say that a bounded linear functional $H \in X^*$  is regular if there exists $h\in X'$ such that  $H(x) = \int_I hx$  for all $x \in X$.  The set of all regular linear functionals from $X^*$ will be denoted by $X_r^*$.
 In the case where $X_a=X_b$ and $X$ has the Fatou property, we have that $(X_a)^*$  is isometric to $X'$, and so $X^* = (X_a)^* \oplus  (X_a)^\perp$ is isometric to $X' \oplus (X_a)^\perp$. The set $(X_a)^\perp$ is called the space of singular functionals and it coincides with those $S\in X^*$ for which $S(x) = 0$ for all $x\in X_a$. It follows that any $F\in X^*$ is represented uniquely as the sum $H+S$  where $H$ is a regular functional and $S$ a singular functional \cite{Z}.

A distribution function $d_x$ of $x \in X$ is defined by $d_x(\lambda) = \mu\{t \in I : |x(t)| > \lambda\}$, $\lambda > 0$, where $\mu = m$ is the Lebesgue measure on $I = [0, \gamma)$, $0 < \gamma \leq \infty$ and the counting measure on $I = \mathbb{N}$.  The decreasing rearrangement of $x$, denoted by $x^*$, is given as $x^*(t) = \inf\{\lambda > 0: d_x(\lambda) \leq t\}$, $t \in [0, \gamma)$. For a sequence $x=(x(i))$, its decreasing rearrangement $x^*$ may be identified with the sequence $ (x^*(i))$ such that $x^*(i) =  \inf\{\lambda > 0: d_x(\lambda) < i\}$ for $i \in \mathbb{N}$. The functions $x, y$ are said to be equimeasurable if $d_x(\lambda) = d_y(\lambda)$ for all $\lambda > 0$, denoted by $x \sim y$. It is clear that $x$ and $x^*$ are equimeasurable. A Banach function lattice $(X, \|\cdot\|)$ is called a rearrangement invariant Banach space if $x \in X$ and $y \in L^0$ with $x \sim y$, we have $y \in X$ and $\|x\| = \|y\|$.

An Orlicz function $\varphi: \mathbb{R}_{+} \rightarrow \mathbb{R}_{+}$ is a convex function such that $\varphi(0) = 0$ and $\varphi(t) > 0$ for $t >0$. It is said to be an Orlicz $N$-function when  $\lim_{t \rightarrow 0}{\varphi(t)}/{t} = 0$ and $\lim_{t \rightarrow \infty} {\varphi(t)}/{t} = \infty$ \cite{Chen}. The complementary function of $\varphi$, denoted by $\varphi_*$, is defined as $\varphi_*(v) = \sup\{uv - \varphi(u): u \geq 0\}$, $v\ge 0$. We have that $\varphi$ is $N$-function if and only if $\varphi_*$ is $N$-function.
Let $p$ and $q$ stand for the right derivatives of $\varphi$ and $\varphi_*$, respectively.  The functions $p$ and $q$ are non-negative, right-continuous and  increasing  on $\mathbb{R}_+$.   If $\varphi$ is $N$-function then  $p(0)= p(0+)= q(0) = q(0+)=0$ and $\lim_{t \rightarrow \infty}p(t)=\lim_{t \rightarrow \infty}q(t)= \infty$. Clearly  for $\varphi$ and $\varphi_*$,  Young's inequality is satisfied, that is, $uv \leq \varphi(u) + \varphi_*(v)$ for all $u,v \in \mathbb{R_+}$.  Recall also that the equality holds  for $v = p(u)$ or $u=q(v)$ \cite{Chen}.

 Let  $w:  I=[0,\gamma)\rightarrow (0, \infty)$ be a weight function that is decreasing and locally integrable. Then we define $W(t): = \int_0^t w < \infty$ for all $t \in I$.  If $\gamma = \infty$,  we assume  $W(\infty) = \infty$.  Given $f \in L^0$,  define the modular
\[
\rho_{\varphi,w}(f) = \int_0^{\gamma} \varphi(f^*(t))w(t)dt = \int_I \varphi(f^*)w.
\]
The modular $\rho_{\varphi,w}$ is orthogonally subadditive,  that is, for $f, g \in L^0$, if $|f| \wedge |g| = 0$, we have $\rho_{\varphi,w}(f + g) \leq \rho_{\varphi,w}(f) + \rho_{\varphi,w}(g)$ \cite{K}. The Orlicz-Lorentz function space $\Lambda_{\varphi,w}$ is the set of all $f \in L^0$ such that $\rho_{\varphi,w}(\lambda f) < \infty$ for some $\lambda > 0$.
It is equipped with either the Luxemburg norm
\[
\|f\| = \|f\|_{\Lambda_{\varphi,w}} = \inf\{\epsilon > 0 : \rho_{\varphi,w}\left({f}/{\epsilon}\right) \leq 1\},
\]
 or the Orlicz norm
\[
\|f\|^0 = \|f\|_{\Lambda_{\varphi,w}}^0 = \sup\left\{\int_I f^*g^*w : \rho_{\varphi_*, w}(g) \leq 1\right\}.
\]
 It is well known that $\|x\| \leq \|x\|^0 \leq 2\|x\|$ \cite{WC, WR}. From now on, we let $\Lambda_{\varphi,w}$ be the Orlicz-Lorentz function space equipped with the Luxemburg norm $\|\cdot\|$ and $\Lambda_{\varphi,w}^0$ be the Orlicz-Lorentz function space equipped with the Orlicz norm $\|\cdot\|^0$. The spaces $\Lambda_{\varphi,w}$ and $\Lambda_{\varphi,w}^0$ are rearrangement invariant Banach spaces. Also, it is well known that $(\Lambda_{\varphi,w})_a = (\Lambda_{\varphi,w})_b = \{x \in \Lambda_{\varphi,w} : \rho_{\varphi,w}(\lambda x) < \infty \ \text{for all} \ \lambda > 0\}$ \cite{K}.

 In the case of sequence spaces let $w = (w(i))$ be a positive decreasing real sequence and $W(n) = \sum_{i=1}^n w(i)$ for all $n \in \mathbb{N}$ and $W(\infty)= \infty$. For a sequence $x \in \ell^0$, we define the modular $\alpha_{\varphi,w}(x)
 = \sum_{i=1}^{\infty} \varphi(x^*(i))w(i) $ and then the Orlicz-Lorentz sequence space $\lambda_{\varphi,w}$ is the set of all real sequences $x= (x(i))$ satisfying $\alpha_{\varphi,w}(\eta x)< \infty$ for some $\eta >0$.  The Luxemburg  and the Orlicz norm on $\lambda_{\varphi,w}$ are defined similarly as in the function case where the modular $\rho_{\varphi,w}$ is replaced by $\alpha_{\varphi,w}$, and    $\lambda_{\varphi, w}$ denotes the Orlicz-Lorentz sequence space equipped with the Luxemburg norm, and $\lambda_{\varphi,w}^0$ with the Orlicz norm.  The both norms are equivalent and the spaces are  rearrangement invariant Banach spaces. We also have $(\lambda_{\varphi,w})_a = (\lambda_{\varphi,w})_b = \{x \in \lambda_{\varphi,w} : \alpha_{\varphi,w}(\eta x) < \infty \ \text{for all} \ \eta >0\}$ \cite{KR2}.

  An Orlicz function $\varphi$ satisfies $\Delta_2$ (resp., $\Delta_2^{\infty}$; $\Delta_2^0$) condition if there exist $K > 0$ (resp., $K>0$ and $u_0\geq 0$; $K>0$ and $u_0 >0$) such that $\varphi(2u) \leq K\varphi(u)$ for all $u \geq 0$ (resp., $u \geq u_0$; $0< u \leq u_0$). {\it Appropriate $\Delta_2$ condition} means $\Delta_2$ and $\Delta_2^\infty$ in the case of the function spaces for $\gamma = \infty$ and $\gamma<\infty$, respectively, and $\Delta_2^0$ for the sequence spaces. It is well known that $(\Lambda_{\varphi,w}^0)_a = \Lambda_{\varphi,w}^0$ and $(\lambda_{\varphi,w}^0)_a = \lambda_{\varphi,w}^0$ if and only if $\varphi$ satisfies the appropriate $\Delta_2$ conditions \cite{K}.

If $f \in \Lambda_{\varphi,w}$ then for some  $\lambda_0 > 0$,  $\rho_{\varphi,w}(\lambda_0 f) < \infty$, and so for any $\lambda > 0$,
$\infty> \rho_{\varphi,w}(\lambda_0 f)  \geq \varphi(\lambda_0 \lambda) \int_0 ^{m\{f^* > \lambda\}} w.$
It follows from $W(\infty) = \infty$ that $d_f(\lambda) = m\{f^* > \lambda\} < \infty$ for every $\lambda > 0$. The similar fact holds for $x\in \lambda_{\varphi,w}$ by $\lambda_{\varphi,w} \subset c_0$.

Let us define $k^* = k^*(f) = \inf\{k>0: \rho_{\varphi_*, w}(p(kf)) \geq 1\}$ and $k^{**} = k^{**}(f) = \sup\{k>0: \rho_{\varphi_*, w}(p(kf)) \leq 1\}$. Clearly $0\le k^* \le k^{**} \le \infty$. If $\varphi$ is $N$-function then $k^{**} < \infty$. Indeed if for a contrary $k^{**} = \infty$, then there exists a non-negative sequence $(k_n)$ such that  $k_n \uparrow \infty$ and $\int_I \varphi_*(p(k_nf)^*)w \leq 1$. Hence for $t_0 = m\{f^* > 1\} < \infty$,
\[
\varphi_*(p(k_n))W(t_0) = \int_0^{t_0}\varphi_*(p(k_n))w =  \int_0^{m\{f^* >1\}} \varphi_*(p(k_n))w \leq \int_I \varphi_*(p(k_n f^*)w \leq 1.
\]
 This implies that $\varphi_*(p(k_n))/p(k_n) \leq 1/W(t_0)p(k_n)$, where the left side tends to $\infty$ since $\varphi_*$ is $N$-function, and the right side approaches $0$ since  $p(k_n) \rightarrow \infty$. This contradiction proves the claim. We define $k^*$ and $k^{**}$ analogously for Orlicz-Lorentz sequence spaces. Set $K(f) = [k^*, k^{**}]$ if $f\in \Lambda_{\varphi,w}$, and similarly $K(x)$ for $x\in\lambda_{\varphi,w}$.

  Recall the following facts which are similar in Orlicz spaces \cite{Chen}.

\begin{Theorem}[\cite{WC}, pg 133]\label{WC}
Let $\varphi$ be an Orlicz N-function. Then,
\begin{enumerate}
\item[$(1)$] If there exists $k>0$ such that $\rho_{\varphi_*,w} (p(kf)) = 1$, then $\|f\|^0 = \int_0^{\gamma} f^*p(kf^*) = \frac{1}{k}(1 + \rho_{\varphi,w} (kf))$.
\item[$(2)$] For any $f \in \Lambda_{\varphi,w}^0$, $\|f\|^0 = \inf_{k>0} \frac{1}{k} (1 + \rho_{\varphi,w}(kf))$.
\item[$(3)$] $k \in K(f)$ if and only if $\|f\|^0 = \frac{1}{k}(1 + \rho_{\varphi,w}(kf))$.
\end{enumerate}
The analogous statements occur in Orlicz-Lorentz sequence space when the modular $\rho_{\varphi,w}$ is replaced by the modular $\alpha_{\varphi,w}$.
\end{Theorem}

 This article has three parts. In section 2, we compute the norm of a singular linear functional $S$ on Orlicz-Lorentz spaces.  We show that $\|S\|$ is the same for both the Luxemburg norm and the Orlicz norm. In section 3, we compute the norm of a bounded linear functional on $\Lambda_{\varphi,w}$ and $\Lambda_{\varphi,w}^0$. The formulas differ dependently on the norm of the space. Furthermore, we show that $(\Lambda_{\varphi,w})_a$ is an $M$-ideal of $\Lambda_{\varphi,w}$, but $(\Lambda_{\varphi,w}^0)_a$ is not an $M$-ideal of $\Lambda_{\varphi,w}^0$ when $\varphi$ is an Orlicz $N$-function and does not satisfy the appropriate $\Delta_2$ condition. The analogous results for the sequence spaces are also given.

\section{Singular linear functionals on Orlicz-Lorentz spaces}

In this section, we show that the formula for $\|S\|$ is the same regardless of Luxemburg or Orlicz norm on Orlicz-Lorentz function or sequence spaces. Letting $f \in L^0$, define $\theta = \theta(f) = \inf\{\lambda > 0 : \rho_{\varphi,w}(f/\lambda) < \infty\}$. It is clear that $\theta(f) < \infty$ for any $f\in \Lambda_{\varphi,w}$. If $f\in (\Lambda_{\varphi,w})_a$, then $\rho_{\varphi,w}\left(\frac{f}{\lambda}\right) < \infty$ for all $\lambda > 0$, so we see that $\theta(f) = 0$. Clearly, $\theta(f) \le \|f\|$. The analogous definitions and facts also hold for Orlicz-Lorentz sequence spaces.

Even though the next two results and their proofs in Orlicz-Lorentz spaces are similar to those in Orlicz spaces \cite{Chen}, we state and prove them in detail because they require slightly different techniques, mostly dealing with decreasing rearrangements.

\begin{Theorem} \label{thm5}
For any $f \in \Lambda_{\varphi,w}$, $\lim_n \|f - f_n\| = \lim_n\|f - f_n\|^0 = \theta(f)$, for $f_n = f \chi_{\{\frac{1}{n} \leq |f| \leq n\}}$.  For any $x = (x(i)) \in \lambda_{\varphi,w}$, $\lim_n\|x-x_n\| = \lim_n \|x - x_n\|^0 = \theta(x)$ for $x_n = x \chi_{\{1,2,...,n\}}$.
\end{Theorem}
\begin{proof}
Let first $f \in (\Lambda_{\varphi,w})_a$. Then, clearly $\theta(f) = 0$. Moreover, in view of $d_f(\lambda) < \infty$ for all $\lambda > 0$,  the functions $f_n = f \chi_{\{\frac{1}{n} \leq |f| \leq n\}}$ are bounded  with supports of finite measure, and $f_n \rightarrow f$ a.e. and $|f_n| \leq |f|$. Since $(\Lambda_{\varphi,w})_a = (\Lambda_{\varphi,w})_b$, from Proposition 1.3.6 in \cite{BS}, we have that $\|f - f_n\| \rightarrow 0$. Moreover, by the equivalence of $\|\cdot \|$ and $\|\cdot\|^0$,  we also get $\|f - f_n\|^0 \rightarrow 0$.

Now, consider $f \in \Lambda_{\varphi,w} \setminus (\Lambda_{\varphi,w})_a$ and $f_n$ as above. In this case, we have $\theta(f) > 0$. Since $d_f(\lambda) < \infty$ for all $\lambda > 0$, and $|f-f_n| \downarrow 0$ a.e., we have $(f- f_n)^* \rightarrow 0$ (\cite{KPS}, pg 68).
Hence $\|f-f_n\|$ and  $\|f-f_n\|^0$ are monotonically decreasing, and so the limits for both $\|f-f_n\|$ and $\|f-f_n\|^0$ exist.

Letting $\epsilon \in (0, \theta)$ we have $\rho_{\varphi,w} \left(\frac{f}{\theta - \epsilon}\right) = \infty$. By the orthogonal subadditivity of $\rho_{\varphi,w}$, we have $\infty = \rho_{\varphi,w} \left(\frac{f}{\theta-\epsilon}\right) \leq \rho_{\varphi,w}\left(\frac{f_n}{\theta - \epsilon}\right) + \rho_{\varphi,w} \left(\frac{f - f_n}{\theta - \epsilon}\right)$. Clearly, the functions $f_n$ are bounded  with supports of finite measure. This implies that $\rho_{\varphi,w}\left(\frac{f_n}{\theta - \epsilon}\right) < \infty$. Hence, we have $\|f - f_n\| \geq \theta - \epsilon$ from the fact that $\rho_{\varphi,w} \left(\frac{f - f_n}{\theta - \epsilon}\right)=\infty$.

On the other hand for $\epsilon > 0$, we have $\rho_{\varphi,w} \left(\frac{f}{\theta + \epsilon}\right) < \infty$ by the definition of $\theta(f)$. Consequently, since $(f - f_n)^* \rightarrow 0$, we get $\lim_{n \rightarrow \infty} \rho_{\varphi,w} \left(\frac{f-f_n}{\theta + \epsilon}\right) = 0$ by the Lebesgue dominated convergence theorem.
Hence, in view of Theorem \ref{WC}.(2), we see that
\[
\|f - f_n\|^0 \leq (\theta + \epsilon) \left(1+ \rho_{\varphi,w} \left(\frac{f - f_n}{\theta + \epsilon}\right)\right) \rightarrow (\theta + \epsilon),
\]
as $n \rightarrow \infty$. Since $\|f\| \leq \|f\|^0$, we finally get
\[
\theta - \epsilon \leq \|f - f_n\| \leq \|f - f_n\|^0 \leq \theta + \epsilon
\]
for sufficiently large $n$ and arbitrary $\epsilon > 0$, and  the proof is complete in the function case. The proof in the sequence case is similar, so we skip it.
\end{proof}

 Now, we compute the norm of a singular functional $S$ on Orlicz-Lorentz function spaces.

\begin{Theorem}\label{theta}
For any singular functional $S$ of $\Lambda_{\varphi,w}$ equipped with the Luxemburg norm or the Orlicz norm, $\|S\| = \|S\|_{(\Lambda_{\varphi,w})^*} = \|S\|_{(\Lambda_{\varphi,w}^0)^*} = \sup\{S(f) : \rho_{\varphi,w}(f) < \infty\} = \sup\{\frac{S(f)}{\theta(f)} : f \in \Lambda_{\varphi,w} \setminus (\Lambda_{\varphi,w})_a\}$.

The analogous formulas hold for Orlicz-Lorentz sequence spaces.

\end{Theorem}

\begin{proof}
Here we also provide the proof only in the function spaces. For a function $f \in \Lambda_{\varphi,w} \setminus (\Lambda_{\varphi,w})_a$, take again $f_n = f \chi_{\{\frac{1}{n} \leq |f| \leq n\}}$. From the fact that $f_n \in (\Lambda_{\varphi,w}^0)_a$ we have $S(f) = S(f- f_n)$ and  $S(f)  \leq \|S\|_{(\Lambda_{\varphi,w}^0)^*} \|f - f_n\|^0$. By Theorem \ref{thm5}, $\|f- f_n\|^0 \rightarrow \theta(f)$,  and so we obtain $\frac{S(f)}{\theta(f)} \leq \|S\|_{(\Lambda_{\varphi,w}^0)^*}$.

If $\rho_{\varphi,w}(f) < \infty$ then $\rho_{\varphi,w}(f - f_n) \rightarrow 0$. Thus for sufficiently large $n$, $\rho_{\varphi,w} (f - f_n) \leq 1$, and  so  $\|f - f_n\| \leq 1$. Hence by Theorem \ref{thm5}, $\theta(f) = \lim_{n \rightarrow \infty} \|f-f_n\| \leq 1$.  Since $S(f) = 0$ for all $f \in (\Lambda_{\varphi,w})_a$, we have $\sup\left\{ S(f) : \rho_{\varphi,w}(f) < \infty\right\} = \sup\left\{ S(f): f \in \Lambda_{\varphi,w} \setminus (\Lambda_{\varphi,w})_a, \  \rho_{\varphi,w}(f) < \infty\right\}$.
Notice that  $S(f) \leq \frac{S(f)}{\theta(f)}$ since $\theta(f) \le 1$. Therefore, taking into account that $\|S\|_{(\Lambda_{\varphi,w}^0)^*} \leq \|S\|_{(\Lambda_{\varphi,w})^*}$ in view of the inequality $\|\cdot\| \le \|\cdot\|^0$ and that $\|f\| \leq 1$ if and only if $\rho_{\varphi,w}(f) \leq 1$,
 we obtain

\begin{eqnarray*}
\|S\|_{(\Lambda_{\varphi,w}^0)^*} \leq \|S\|_{(\Lambda_{\varphi,w})^*} &=& \sup\{S(f) : \rho_{\varphi,w}(f) \leq 1\}\\
&\leq& \sup\left\{ S(f) : \rho_{\varphi,w}(f) < \infty\right\}\\
&\leq& \sup\left\{ \frac{S(f)}{\theta(f)} : f \in \Lambda_{\varphi,w} \setminus (\Lambda_{\varphi,w})_a, \,\,\, \rho_{\varphi,w}(f) < \infty\right\}\\
&\leq& \sup\left\{ \frac{S(f)}{\theta(f)} : f \in \Lambda_{\varphi,w} \setminus (\Lambda_{\varphi,w})_a\right\}\\
&\leq& \|S\|_{(\Lambda_{\varphi,w}^0)^*}.
\end{eqnarray*}

\end{proof}

\section{Norm of bounded linear functionals}

We need to recall first the K\"othe associate space to an Orlicz-Lorentz space. For any non-negative integrable function $f\in L^0$ and $0\le a < b < \infty$, denote $F(a,b) = \int_a^b f$.  Let $h \in L^0$  be non-negative and locally integrable on $I$. Then   the interval $(a, b) \subset I$ is called a level interval of $h$ with respect to the weight $w$, if $R(a,t) := \frac{H(a,t)}{W(a,t)} \leq \frac{H(a,b)}{W(a,b)} = R(a, b)$ for all $a < t < b$ and $R(a,b) > 0$.  In the case where $b = \infty$, define $R(a,b) = R(a, \infty) = \limsup_{t \rightarrow \infty}R(a, t)$.  If the level interval $(a,b)$ is not contained in a larger level interval, we say that $(a,b)$ is a maximal level interval. Halperin's level function of $h$, denoted by $h^0$, is defined as
\[
h^0(t) =
\begin{cases}
      R(a_j, b_j)w(t) = \frac{H(a_j, b_j)}{W(a_j, b_j)}w(t) , & t \in (a_j, b_j) \ \ \text{for some} \ \ j, \\
      h(t), & t \notin \cup_j(a_j, b_j),
   \end{cases}
\]
\noindent provided that each $(a_j, b_j)$ is a maximal level interval. Similarly, for a non-negative sequence $h= (h(i)) \in l^0$ and a positive decreasing weight $w= (w(i))$, the interval $(a, b] = \{a+1, a+2, ... , b\}$ is called a level interval if $r(a,j) = \frac{h(a,j)}{w(a,j)} \leq \frac{h(a,b)}{w(a,b)} = r(a,b)$ for every $a+1 \leq j \leq b$ and $r(a,b) >0$, where $h(a,j) = \sum_{i=a+1}^j h(i)$ and $w(a, j) = \sum_{i=a+1}^j w(i)$. The level sequence $h^0$ is defined as
 \[
h^0(i) =
\begin{cases}
      r(a_j, b_j)w(i) , & i \in (a_j, b_j] \ \ \text{for some} \ \ j , \\
      h(i), & i \notin \cup_j(a_j, b_j],
   \end{cases}
\]
where each $(a_j, b_j]$ is a maximal level interval.
Letting $h\in L^0$ define $P_{\varphi,w}(h) = \inf\left\{\int_I \varphi\left(\frac{|h|}{v}\right)v : v \prec w\right\} $, and then
the space $\mathcal{M}_{\varphi,w}$ as the set of all $h \in L^0$ such that $P_{\varphi,w}(\lambda h) < \infty$ for some $\lambda > 0$. By Theorem 4.7 in \cite{KLR} we have $P_{\varphi,w}(h) = \int_I \varphi((h^*)^0/w) w$ if $\varphi$ is $N$-function. The Luxemberg norm and the Orlicz norm for the modular $P_{\varphi,w}$ are defined as,
\[
\|h\|_{\mathcal{M}_{\varphi,w}} = \inf\{\epsilon > 0 : P_{\varphi,w}\left({h}/{\epsilon}\right) \leq 1\}
\ \ \ \text{and} \ \ \  
\|h\|_{\mathcal{M}_{\varphi,w}^0} = \inf_{k > 0} \frac{1}{k}(1 + P_{\varphi,w}(kh)),
\]
respectively.

For $h \in \ell^0$, we define
$p_{\varphi,w}(h) = \inf\left\{\sum_{i=1}^{\infty} \varphi\left(\frac{|h(i)|}{v(i)}\right)v(i) : v \prec w \right\}$.
The space $\mathfrak{m}_{\varphi,w}$ is the set of all $h = (h(i))$ such that $p_{\varphi,w}(\eta h) < \infty$ for some $\eta > 0$. The Luxemburg norm and the Orlicz norm on $\mathfrak{m}_{\varphi,w}$  are given analogously as in function spaces where we replace $P_{\varphi,w}$ by  $p_{\varphi,w}$.

 From now on  we denote by $\mathcal{M}_{\varphi,w}$ and $\mathfrak{m}_{\varphi,w}$ the space equipped with the Luxemburg norm $\| \cdot \|_{\mathcal{M}_{\varphi,w}}$ and $\|\cdot\|_{\mathfrak{m}_{\varphi,w}}$ respectively, and $\mathcal{M}_{\varphi,w}^0$ and $\mathfrak{m}_{\varphi,w}^0$ the space equipped with the Orlicz norms $\| \cdot \|_{\mathcal{M}_{\varphi,w}^0}$ and $\|\cdot\|_{\mathfrak{m}_{\varphi,w}^0}$ respectively. All those spaces  are  rearrangement invariant Banach spaces \cite{KR}.

\begin{Theorem}[\cite{KLR}, Theorems 2.2, 5.2]\label{th:01}
Let  $w$ be a decreasing weight and $\varphi$ be an Orlicz $N$-function.  Then the K\"othe dual space to an Orlicz-Lorentz space $\Lambda _{\varphi ,w}$ (resp. $\Lambda_{\varphi,w}^0$)  is expressed as
\begin{equation*}
\left( \Lambda _{\varphi ,w}\right) ^{\prime }=\mathcal{M}_{\varphi _{\ast
},w}^{0}\ \ \ (\text{resp.}\ (\Lambda_{\varphi,w}^0)^{\prime} = \mathcal{M}_{\varphi _{\ast
},w})
\end{equation*}
with equality of norms.
Similarly in the sequence case we have
\begin{equation*}
\left(\lambda _{\varphi ,w}\right) ^{\prime }=\mathfrak{m}_{\varphi _{\ast},w}^{0}\ \ (\text{resp.} \ \ (\lambda_{\varphi,w}^0)' = \mathfrak{m}_{\varphi_*,w})
\end{equation*}
with equality of norms.

\end{Theorem}

Let $X$ be an Orlicz-Lorentz function or sequence space equipped with either norm. Then, $X^* = X_r \oplus X_s$, where $X_r$ is isomorphically isometric to its K\"othe associate space $X'$, and $X_s = (X_a)^\perp$.

\begin{Theorem}\label{th:lux}
Assume $\varphi$ is $N$-function. Let $F$ be a bounded linear functional on $\Lambda_{\varphi,w}$. Then $F= H + S$, where $H(f) = \int_I fh$ for some
$h\in \mathcal{M}^0_{\varphi_*,w}$,
$\|H\|= \|h\|^0_{\mathcal{M}_{\varphi_*,w}}$, $S(f)=0$
 for all $f\in (\Lambda_{\varphi,w})_a$, and     $\|F\|_{(\Lambda_{\varphi,w})^*} = \|h\|_{\mathcal{M}_{\varphi_*, w}}^0 + \|S\|$.
\end{Theorem}

\begin{proof}

By Theorem \ref{th:01} and the remark above, $F= H+S$ uniquely, where $H(f) = \int_I hf$ for some $h\in \mathcal{M}^0_{\varphi_*,w}$ with
$\|H\|= \|h\|^0_{\mathcal{M}_{\varphi_*,w}}$, and $S(f)=0$
 for all $f\in (\Lambda_{\varphi,w})_a$. Observe  by Theorem \ref{theta} that the norm of the singular functional  $\|S\|$ is the same under either the Luxemburg norm or the Orlicz norm.

Clearly  $\|F\|_{(\Lambda_{\varphi,w})^*} = \|H+S\|_{(\Lambda_{\varphi,w})^*} \leq \|H\|_{(\Lambda_{\varphi,w})^*} + \|S\| = \|h\|_{\mathcal{M}_{\varphi_*, w}}^0 + \|S\|$. Now we show the opposite inequality.

Let $\epsilon >0$ be arbitrary. From the definitions of $\|h\|_{\mathcal{M}_{\varphi_*,w}}^0$ and $\|S\|$, we can choose $f, g \in \Lambda_{\varphi,w}$ with $\|f\| \leq 1, \|g\| \leq 1$ such that

\begin{equation} \label{DN}
\|h\|_{\mathcal{M}_{\varphi_*,w}}^0 - \epsilon < \int_I hf \,\,\, \text{and} \,\,\, \|S\| - \epsilon < S(g).
\end{equation}

We can assume that $f$ is bounded. Indeed, let $z \in S_{\Lambda_{\varphi,w}}$ be such that $\|h\|_{\mathcal{M}_{\varphi_*,w}}^0- \frac{\epsilon}{2} < \int_I |hz|$. Let $(z_n)_{n=1}^{\infty}$ be a sequence of non-negative bounded functions with supports of finite measure defined on $[0,n)$ such that $z_n \uparrow |z|$ a.e.  Then, $\int_I |h| |z| = \lim_{n \rightarrow \infty} \int_I |h| z_n$ by the monotone convergence theorem, which implies that for all $\epsilon > 0$, there exists $z_{n_0}$ such that $\int_I |hz| - \frac{\epsilon}{2} \leq \int_I |h| z_{n_0}$. Hence,

\[
\|h\|_{\mathcal{M}_{\varphi_*, w}}^0 - \frac{\epsilon}{2} - \frac{\epsilon}{2} \leq \int_I |hz| - \frac{\epsilon}{2} \leq \int_I |h|z_{n_0}.
\]

\noindent Let $f = (\sign{h})z_{n_0}$. Thus, we found a bounded function $f$ of support of finite measure such that $\|f\| \leq 1$ and $\|h\|_{\mathcal{M}_{\varphi_*, w}}^0 - \epsilon < \int_I hf$.\

Since $H$ is a bounded linear functional on $\Lambda_{\varphi,w}$, $hf$ is integrable, so there exists $\delta > 0$ such that for every measurable subset $E \subset I$, with $mE < \delta$, we have

\begin{equation}\label{Ex}
\int_{E} |hf| < \epsilon.
\end{equation}

Now, we show that there exist $n \in \mathbb{N}$ and a measurable subset $E \subset I$ such that $mE < \delta$ and  

\begin{equation}\label{Ey}
\int_E |hg| < \epsilon,\ \ \int_0^{mE} \varphi(g^*) w < \frac{\epsilon}{2},\ \ \int_I \varphi((g\chi_{[n, \gamma)})^*)w < \frac{\epsilon}{2},\ \ \text{and} \ \   \int_n^{\gamma} |hg|< \epsilon.
\end{equation}

\noindent Indeed, let $E_n = \{g^* > n\} = [0, t_n)$ and define $g_n^* = g^* \chi_{[0, t_n)}$. We see that $g_n^* \leq g^*$ and $ g_n^* \downarrow 0$ a.e., so by the Lebesgue dominated convergence theorem, $\lim_{n \rightarrow \infty} \int_I \varphi(g_n^*) w = 0$. This implies that for any $\epsilon >0$, there exists $N_1$ such that for every $n \geq N_1$,

\begin{equation} \label{Eyn}
\int_I \varphi(g_n^*)w = \int_I \varphi(g^*\chi_{[0, t_n)})w = \int_0^{t_n} \varphi(g^*)w = \int_0^{mE_n} \varphi(g^*)w<\frac{\epsilon}{2}.
\end{equation}

\noindent Also, $E_{n+1} \subset E_n$ for all $n \in \mathbb{N}$ and $m(\cap E_n) = m\{g^* =\infty\} =  0$. By continuity of measure, $0= m(\cap E_n) = \lim_{n \rightarrow \infty} m\{g^*>n\}$.\

Since $g \sim g^*$, we see that $\lim_{n \rightarrow \infty} m\{|g| > n\} = 0$. The function $hg$ is integrable, so we have $\lim_{n \rightarrow \infty} \int_{\{|g| > n\}} |hg| = 0$. Then, there exists $N_2$ such that $\int_{\{|g| > n\}} |hg| < \epsilon$ for $n \geq N_2$. Since $\rho_{\varphi,w}(g) < \infty$, we choose sufficiently large $n \geq N = \max\{N_1, N_2\}$ satisfying $mE_n = m \{|g|>n\} < \delta$, $\supp{f} \cap [n, \gamma) = \emptyset$, $\int_I \varphi((g\chi_{[n, \gamma)})^*)w < \frac{\epsilon}{2}$, and $\int_{[n,\gamma)} |hg|< \epsilon$. By letting $E = \{|g|>n\}$ for such $n$, we found $n \in \mathbb{N}$ and a measureable subset $E \subset I$ satisfying (\ref{Ey}). Note that $\supp{f} \subset [0, n)$ from the construction.

Define

\[u(t) =
\begin{cases}
f(t), & t \in G_1 = \supp{f} \setminus E\\
g(t), & t \in G_2 = E \cup [n, \gamma)\\
0, &  \text{otherwise}.
\end{cases}
\]

\noindent By the orthogonal subadditivity of the modular $\rho_{\varphi, w}$, we have

\begin{eqnarray*}
\rho_{\varphi,w}(u) = \int_I \varphi(f \chi_{G_1} + g \chi_{G_2})^*w &\leq& \int_I \varphi((f\chi_{G_1})^*)w +   \int_I \varphi((g\chi_ {G_2})^*)w\\
&\leq& \int_0^{mG_1} \varphi(f^*)w + \int_I \varphi(g\chi_E + g\chi_{[n, \gamma) \setminus E})^* w \\
&\leq&  \int_0^{mG_1} \varphi(f^*)w +  \int_I \varphi(g\chi_E)^* w +  \int_I \varphi(g\chi_{[n, \gamma)\setminus E})^* w\\
&\leq&  \int_0^{mG_1} \varphi(f^*)w +  \int_0^{mE} \varphi(g^*) w + \int_I \varphi(g\chi_{[n, \gamma)})^* w\\
&\leq& 1 + \epsilon,
 \end{eqnarray*}

\noindent which implies that $\rho_{\varphi,w}(\frac{u}{1+\epsilon}) \leq 1$, and so $\|\frac{u}{1+\epsilon}\| \leq 1$. We see that $S(f\chi_{G_1}) = 0$ from $f\in (\Lambda_{\varphi,w})_a$. Also, $g \chi_{G_1} \in (\Lambda_{\varphi,w})_a$ because $mG_1 = m(\supp{f} \setminus E) \leq m([0,n) \setminus E) < \infty$ and $g$ is bounded on $G_1$. This implies that $S(g \chi_{G_1}) = 0$. Hence,  $S(g) = S(g \chi_{G_1}) + S(g \chi_{G_2}) = S(g \chi_{G_2})$. Moreover, from (\ref{Ey}), we have $\left|\int_{E \setminus [n, \gamma)} hg\right| \leq \int_{E \setminus [n, \gamma)} |hg| \leq \int_E |hg| < \epsilon$.

It follows that

\begin{eqnarray*}
(1+ \epsilon)\|F\| \geq (1+ \epsilon) F\left(\frac{u}{1+\epsilon}\right) = F(u) &=& F(f \chi_{G_1} + g \chi_{G_2})\\
&=& \int_I h(f \chi_{G_1}+ g\chi_{G_2}) + S((f \chi_{G_1} + g\chi_{G_2})) \\
&=& \int_I hf \chi_{G_1} + \int_I hg\chi_ {G_2} +S(f \chi_{G_1}) + S(g\chi_{G_2})  \\
&=& \int_{\supp{f} \setminus E} hf +\int_{E \cup [n, \gamma)} hg + S(g\chi_{G_2})\\
&=& \int_I hf -\int_E hf +\int_{E \setminus [n, \gamma)} hg + \int_{[n, \gamma)} hg + S(g)\\
&>& \|h\|_{\mathcal{M}_{\varphi_*,w}}^0 - 2\epsilon - 2\epsilon + S(g)\,\,\, (\text{by (\ref{Ex}) and (\ref{Ey})})\\
&>& \|h\|_{\mathcal{M}_{\varphi_*, w}}^0 -2\epsilon -2\epsilon + \|S\| - \epsilon \,\,\, (\text{by (\ref{DN})})\\
&=& \|h\|_{\mathcal{M}_{\varphi_*, w}}^0 + \|S\| - 5\epsilon.
\end{eqnarray*}

\noindent As $\epsilon \rightarrow 0$, the proof is done.
\end{proof}

The sequence version below has analogous  (simpler) proof so we skip it.

\begin{Theorem}\label{th:luxseq} Suppose $\varphi$ is $N$-function and let $F$ be a bounded linear functional on $\lambda_{\varphi,w}$. Then $F= H+S$, where $H(x) = \sum_{i=1}^{\infty} x(i) y(i)$, $\|H\| = \| y\|^0_{\mathfrak{m}_{\varphi_*,w}}$, $S$ is a singular functional vanishing on $(\lambda_{\varphi,w})_a$ and $\|F\|_{(\lambda_{\varphi,w})^*} = \|y\|^0_{\mathfrak{m}_{\varphi_*,w}} + \|S\|$.
\end{Theorem}

As a consequence of Theorems \ref{th:lux} and \ref{th:luxseq} we obtain the following result.

\begin{Corollary}\label{cor:luxideal}
If $\varphi$ does not satisfy the appropriate $\Delta_2$ condition then the order-continuous subspaces $(\Lambda_{\varphi,w})_a$ and $(\lambda_{\varphi,w})_a $ are non-trivial $M$-ideals of $\Lambda_{\varphi,w}$ and $\lambda_{\varphi,w}$, respectively.
\end{Corollary}

Recall \cite{KR,  KLR} that for an Orlicz $N$-function $\varphi$ and $h \in L^0$ we have
\begin{equation}\label{form:1}
P_{\varphi,w}(h) = \inf\left\{\int_I \varphi\left(\frac{h^*}{v}\right)v : v \prec w, v\downarrow\right\}=\int_I \varphi \left(\frac{(h^*)^0}{w}\right) w,
\end{equation}
and that similar formula holds true for any sequence $x\in\ell^0$ \cite{KLR}. Hence, we have

\[
p_{\varphi,w}(h) = \sum_{i=1}^{\infty} \varphi\left(\frac{(h^*)^0(i)}{w(i)}\right)w(i).
\]

Consider the  decreasing simple function  $h^* = \sum_{i=1}^n a_i \chi_{(t_{i-1}, t_i)}$ where $ a_1 > a_2  > \cdots > a_n >0$ and $t_0 = 0$. Let $H^*(a,b) = \int_a^b h^*$. By Algorithm A provided in \cite{KLR}, the maximal level intervals of $h^*$ are of the form $(t_{i_j}, t_{i_{j+1}})$ where $(t_{i_j})_{j=0}^{l-1}$ is a subsequence of $(t_i)_{i=1}^n$ with $0= t_0 = t_{i_0}< t_{i_1} < ... <  t_{i_{l}} = t_n < \infty$. Then, we have

\begin{equation}\label{level}
\frac{(h^*)^0}{w} = \frac{\sum_{j=0}^{l-1}R(t_{i_j}, t_{i_{j+1}})w\chi_{(t_{i_j}, t_{i_{j+1}})}}{w} = \sum_{j=0}^{l-1} R(t_{i_j}, t_{i_{j+1}}) \chi_{(t_{i_j}, t_{i_{j+1}})} =\sum_{j=0}^{l-1} \frac{H^*(t_{i_j}, t_{i_{j+1}})}{W(t_{i_j}, t_{i_{j+1}})}\chi_{(t_{i_j}, t_{i_{j+1}})} .
\end{equation}

Observe that the sequence $(R(t_{i_j}, t_{i_{j+1}}))_{j=0}^{l-1}$ is decreasing since $\frac{(h^*)^0}{w}$ is decreasing (\cite{Hal}, Theorem 3.6).  Furthermore, we obtain
\[
P_{\varphi,w}(h) = \int_I \varphi \left(\sum_{j=0}^{l-1} \frac{H^*(t_{i_j}, t_{i_{j+1}})}{W(t_{i_j}, t_{i_{j+1}})}\chi_{(t_{i_j}, t_{i_{j+1}})}\right) w = \sum_{j=0}^{l-1} \varphi \left(\frac{H^*(t_{i_j}, t_{i_{j+1}})}{W(t_{i_j}, t_{i_{j+1}})} \right) \cdot W(t_{i_j}, t_{i_{j+1}}).
\]
The next lemma is a key ingredient for computation of the norm of a bounded linear functional on $\Lambda_{\varphi,w}^0$ or $\lambda^0_{\varphi,w}$.

\begin{Lemma}\label{lem3}
Let $h \in L^0$ be a non-negative simple function with support of finite measure. Then, there exists a non-negative simple function $v$ such that
\begin{equation*}
P_{\varphi_*, w}(h) = \int_I \varphi_* \left(\frac{h}{v}\right)v \,\,\, \text{and} \,\,\, \int_I \varphi\left(q\left(\frac{h}{v}\right)\right) v = \int_I \varphi\left(q\left(\frac{h}{v}\right)^*\right) w.
\end{equation*}

The similar formula holds for modular $p_{\varphi_*, w} (x)$ for any $x\in \ell^0$.

\end{Lemma}

\begin{proof}
Let $h = \sum_{i=1}^n a_i \chi_{A_i}$ with $ a_1 > a_2  > \cdots > a_n >0$  and $\{A_i\}_{i=1}^n$ be a family of disjoint  measurable subsets of $I$ with finite measure. Since $h$ and $h^*$ are equimeasurable, we see that $mA_i =  t_i- t_{i-1}$ for $i=1,\dots,n$.

 It is well known by \cite{Hal} and \cite{KLR} that each $(t_{i-1}, t_i)$ is a level interval of $h^*$, contained in at most one maximal level interval $(t_{i_j}, t_{i_{j+1}})$ for some $0 \leq j \leq l-1$ \cite{KLR}. So, for every $j$, we can see

\[
m(t_{i_j}, t_{i_{j+1}}) = m(\cup_{i_j < i \leq i_{j+1}}(t_{i-1}, t_i)) = m(\cup_{i_j < i \leq i_{j+1}}A_i),
\]

\noindent and this implies

\begin{equation}\label{tc}
H^*(t_{i_j}, t_{i_{j+1}}) = \int_{t_{i_j}}^{t_{i_{j+1}}} h^* = \sum_{i = i_j+1}^{i_{j+1}} \int_{t_{i-1}}^{t_i} a_i = \sum_{i = i_j+1}^{i_{j+1}} a_i(t_i - t_{i-1}) = \sum_{i_j < i \leq i_{j+1}} a_i mA_i.
\end{equation}

\noindent By (\ref{level}), we have
\begin{equation*}
\frac{(h^*)^0}{w} =\sum_{j=0}^{l-1} \sum_{i_j < i \leq i_{j+1}} \frac{H^*(t_{i_j}, t_{i_{j+1}})}{W(t_{i_j}, t_{i_{j+1}})} \chi_{(t_{i-1}, t_i)} =  \left(\sum_{j=0}^{l-1} \sum_{i_j < i \leq i_{j+1}} \frac{H^*(t_{i_j}, t_{i_{j+1}})}{W(t_{i_j}, t_{i_{j+1}})} \chi_{A_i}\right)^*.
\end{equation*}

\noindent Hence, by right-continuity of $q$, we also have $q\left(\frac{(h^*)^0}{w}\right) =  q\left(\sum_{j=0}^{l-1} \sum_{i_j < i \leq i_{j+1}} \frac{H^*(t_{i_j}, t_{i_{j+1}})}{W(t_{i_j}, t_{i_{j+1}})} \chi_{A_i} \right)^*$. Let $v= \sum_{j=0}^{l-1} \sum_{i_j < i \leq i_{j+1}} \frac{W(t_{i_j}, t_{i_{j+1}})}{H^*(t_{i_j}, t_{i_{j+1}})} a_i \chi_{A_i}$. Then, $q\left(\frac{(h^*)^0}{w}\right) = q\left(\frac{h}{v}\right)^*$. The functions $h$ and $v$ have the same supports, so the quotient $h/v$ is set to be zero outside of the supports of $h$ and $v$.

Now, we compute $\int_I \varphi_* \left(\frac{h}{v}\right)v$ and $\int_I \varphi\left(q\left(\frac{h}{v}\right)\right) v$.

\begin{eqnarray*}
\int_I \varphi_* \left(\frac{h}{v}\right)v &=& \int_I \varphi_* \left(\frac{\sum_{i=1}^n a_i \chi_{A_i}}{ \sum_{j=0}^{l-1} \sum_{i_j < i \leq i_{j+1}} \frac{W(t_{i_j}, t_{i_{j+1}})}{H^*(t_{i_j}, t_{i_{j+1}})} a_i \chi_{A_i}}\right) \cdot  \sum_{j=0}^{l-1} \sum_{i_j < i \leq i_{j+1}} \frac{W(t_{i_j}, t_{i_{j+1}})}{H^*(t_{i_j}, t_{i_{j+1}})} a_i \chi_{A_i}\\
&=&  \sum_{j=0}^{l-1} \sum_{i_j < i \leq i_{j+1}} \int_I \varphi_* \left(\frac{H^*(t_{i_j}, t_{i_{j+1}})}{W(t_{i_j}, t_{i_{j+1}})} \right) \cdot \frac{W(t_{i_j}, t_{i_{j+1}})}{H^*(t_{i_j}, t_{i_{j+1}})} a_i \chi_{A_i}\\
&=& \sum_{j=0}^{l-1} \varphi_* \left(\frac{H^*(t_{i_j}, t_{i_{j+1}})}{W(t_{i_j}, t_{i_{j+1}})} \right) \cdot \frac{W(t_{i_j}, t_{i_{j+1}})}{H^*(t_{i_j}, t_{i_{j+1}})} \sum_{i_j < i \leq i_{j+1}} a_i \cdot mA_i\\
&=& \sum_{j=0}^{l-1} \varphi_* \left(\frac{H^*(t_{i_j}, t_{i_{j+1}})}{W(t_{i_j}, t_{i_{j+1}})} \right) \cdot W(t_{i_j}, t_{i_{j+1}})\,\,\, \text{(by (\ref{tc}))}\\
&=& P_{\varphi_*,w}(h).
\end{eqnarray*}

\noindent and
\begin{eqnarray*}
\int_I \varphi\left(q\left(\frac{h}{v}\right)\right)v &=& \int_I \varphi\left(q\left(\frac{\sum_{i=1}^n a_i \chi_{A_i}}{ \sum_{j=0}^{l-1} \sum_{i_j < i \leq i_{j+1}} \frac{W(t_{i_j}, t_{i_{j+1}})}{H^*(t_{i_j}, t_{i_{j+1}})} a_i \chi_{A_i}}\right)\right) \cdot  \sum_{j=0}^{l-1} \sum_{i_j < i \leq i_{j+1}} \frac{W(t_{i_j}, t_{i_{j+1}})}{H^*(t_{i_j}, t_{i_{j+1}})} a_i \chi_{A_i}\\
&=& \sum_{j=0}^{l-1} \sum_{i_j < i \leq i_{j+1}} \int_I \varphi \left(q\left(\frac{H^*(t_{i_j}, t_{i_{j+1}})}{W(t_{i_j}, t_{i_{j+1}})} \right)\right) \cdot \frac{W(t_{i_j}, t_{i_{j+1}})}{H^*(t_{i_j}, t_{i_{j+1}})} a_i \chi_{A_i}\\
&=& \sum_{j=0}^{l-1} \varphi\left(q\left(\frac{H^*(t_{i_j}, t_{i_{j+1}})}{W(t_{i_j}, t_{i_{j+1}})} \right)\right) \cdot \frac{W(t_{i_j}, t_{i_{j+1}})}{H^*(t_{i_j}, t_{i_{j+1}})} \sum_{i_j < i \leq i_{j+1}} a_i \cdot mA_i\\
&=& \sum_{j=0}^{l-1} \int_I \varphi\left(q\left(\frac{H^*(t_{i_j}, t_{i_{j+1}})}{W(t_{i_j}, t_{i_{j+1}})} \right)\right) \cdot w\chi_{(t_{i_j}, t_{i_{j+1}})} \,\,\, (\text{by (\ref{tc})}) \\
&=& \int_I \varphi\left(q\left(\sum_{j=0}^{l-1} \frac{H^*(t_{i_j}, t_{i_{j+1}})}{W(t_{i_j}, t_{i_{j+1}})}\chi_{(t_{i_j}, t_{i_{j+1}})}\right)\right) w\\
&=& \int_I \varphi \left(q \left(\frac{(h^*)^0}{w}\right) \right) w = \int_I \varphi \left(q\left(\frac{h}{v}\right)^*\right)w.
\end{eqnarray*}

\end{proof}

Now, we are ready to compute the norm of a bounded linear  functional  in $\Lambda_{\varphi,w}^0$.

\begin{Theorem}\label{Orlicz}
Let $\varphi$ be an Orlicz $N$-function and $F$ be a bounded linear functional on $\Lambda^0_{\varphi,w}$. Then $F= H + S$, where $H(f) = \int_I fh$ for some
$h\in \mathcal{M}_{\varphi_*,w}$,
$\|H\|= \|h\|_{\mathcal{M}_{\varphi_*,w}}$, $S(f)=0$
 for all $f\in (\Lambda_{\varphi,w})_a$, and   $\|F\| = \inf\{\lambda>0 : P_{\varphi_*,w}(\frac{h}{\lambda}) + \frac{1}{\lambda}\|S\| \leq 1\}$.
\end{Theorem}

\begin{proof}
Similarly as in Theorem \ref{th:lux}, we have $F = H + S$, where $H(f) = \int_I hf$ for some $h \in \mathcal{M}_{\varphi_*,w}$ with $\|H\| = \|h\|_{\mathcal{M}_{\varphi_*,w}}$ and $S(f) = 0$ for all $f \in (\Lambda_{\varphi,w}^0)_a$ in view of Theorem \ref{th:01}. Thus, we only need to show the formula for $\|F\|$. Without loss of generality, assume $\|F\| = 1$. Let $f \in S_{\Lambda_{\varphi,w}^0}$. Since $h \in \mathcal{M}_{\varphi_*,w}$, we have $P_{\varphi_*,w}(\frac{h}{\lambda}) < 1$ for some $\lambda>0$. So, we can choose $\lambda>0$ such that $P_{\varphi_*,w}(\frac{h}{\lambda})+ \frac{1}{\lambda} \|S\| \leq 1$. Let $k \in K(f)$. By Theorem \ref{WC}.(3), $1 = \|f\|^0 = \frac{1}{k}(1 + \rho_{\varphi,w}(kf))$, and this implies that $\rho_{\varphi,w}(kf) < \infty$. For every $v \prec w,\, v \downarrow$, we have

\[
\frac{1}{\lambda}F(kf) = \frac{1}{\lambda}\left(\int_I khf + S(kf)\right) \leq \frac{1}{\lambda} \left(\int_I kh^*f^* + S(kf)\right) = \int_I \frac{kh^*f^*v}{\lambda v} + \frac{1}{\lambda} S(kf).
\]

\noindent By Young's inequality, we see that $\int_I \frac{kh^*f^*v}{\lambda v} + \frac{1}{\lambda} S(kf) \leq  \int_I \varphi(kf^*)v + \int_I \varphi_*\left(\frac{h^*}{\lambda v}\right)v + \frac{1}{\lambda}S(kf)$. Since by (\ref{form:1}) this is for all $v \prec w$, $v\downarrow$, and by Hardy's lemma (\cite{BS}, Proposition 3.6), we get

\[
\frac{1}{\lambda}F(kf) \leq \int_I \varphi(kf^*)v + \int_I \varphi_*\left(\frac{h^*}{\lambda v}\right)v + \frac{1}{\lambda}S(kf) \leq\rho_{\varphi,w}(kf) + P_{\varphi_*,w}\left(\frac{h}{\lambda}\right) + \frac{1}{\lambda}S(kf)).
\]

\noindent Furthermore, $S(kf) \leq \|S\|$ because $\rho_{\varphi,w}(kf) < \infty$. Hence,

\[
\frac{1}{\lambda}F(kf) \leq \rho_{\varphi,w}(kf) + P_{\varphi_*,w}\left(\frac{h}{\lambda}\right) + \frac{1}{\lambda}\|S\| \leq 1+ \rho_{\varphi,w}(kf),
\]

\noindent which implies that $F(f) \leq \lambda \cdot \frac{1}{k}(1 + \rho_{\varphi,w}(kf)) \leq \lambda \|f\|^0=\lambda$. Since $f$ and $\lambda$ are arbitrary, we showed that $\|F\| \leq \inf\{\lambda>0 : P_{\varphi_*,w}(\frac{h}{\lambda}) + \frac{1}{\lambda} \|S\| \leq 1 \}$.

Now, suppose that

\[
1= \|F\| < \inf\{\lambda>0 : P_{\varphi_*,w}\left(\frac{h}{\lambda}\right) + \frac{1}{\lambda}\|S\| \leq 1\}.
\]

\noindent Then, there exists $\delta>0$ such that

\[
 P_{\varphi_*,w}(h) + \|S\| > 1 + 3\delta.
\]

\noindent From Theorem \ref{theta}, $\|S\| = \sup\{S(f): \rho_{\varphi,w}(f) < \infty\}$. So, there exists $f \in \Lambda_{\varphi,w}^0$ such that $\rho_{\varphi,w}(f) < \infty$ and  $\|S\| <  S(f) + \delta$. This implies that

\[
 P_{\varphi_*,w}(h) + S(f)+ \delta >P_{\varphi_*,w}(h) + \|S\| > 1 + 3\delta,
\]

\noindent and so

\[
P_{\varphi_*,w}(h) + S(f) > 1 + 2\delta.
\]

Without loss of generality, let $h \geq 0$. Let $(h_n)_{n=1}^{\infty}$ be a sequence of simple functions with support of finite measure such that $h_n \uparrow h$. By Lemma 4.6 in \cite{KR}, we get $P_{\varphi_*,w}(h_n) \uparrow P_{\varphi_*,w}(h)$. Hence, there exists a non-negative simple function $h_0$ with $m(\supp{h_0}) < \infty$ such that $0\le h_0 \le h$ a.e. and

\[
P_{\varphi_*,w}(h) < P_{\varphi_*,w}(h_0) + \delta.
\]

\noindent This implies that

\[
P_{\varphi_*,w}(h_0) + S(f) > P_{\varphi_*,w}(h) + S(f) - \delta > 1 + 2\delta - \delta = 1 + \delta.
\]

Now, consider a function $f_n = f \chi_{\{\frac{1}{n}\leq |f| \leq n\}}$. The function $|f-f_n| \downarrow 0$ a.e. Hence, we have $(f- f_n)^* \rightarrow 0$, and so $\rho_{\varphi,w}(f-f_n) \downarrow 0$ by the Lebesgue dominated convergence theorem.
Since $H$ is a bounded linear functional on $\Lambda_{\varphi,w}^0$, we have $\int_I |f-f_n| h \le \int_I |f| h < \|H\|\|f\|^0< \infty$, and so $\int_I |f-f_n| h \to 0$. For $\delta > 0$, there exists $N_0$  such that for $n\ge N_0$, we have

\[
\rho_{\varphi,w}(f-f_n) \le 1  \ \ \ \text{and} \ \ \
\int_I |f-f_n| h < \frac{\delta}{8}.
\]

\noindent Let $g_1 = f-f_n$ for some $n\ge N_0$. The function $f_n$ is bounded with support of finite measure since $\supp{f_n} \subset \{|f| > \frac{1}{n}\}$ and $m\{|f| > \frac{1}{n}\} < \infty$. Thus, we have $S(f) = S(g_1) + S(f_n) = S(g_1)$ and

\begin{equation}\label{eq1}
\rho_{\varphi,w}(g_1) \leq 1, \,\,\, \int_I |g_1| h < \frac{\delta}{8}, \  \ \text{and} \,\,\, P_{\varphi_*,w}(h_0) + S(g_1) > 1 + \delta.
\end{equation}

Let $v$ be the non-negative simple function constructed in Lemma \ref{lem3} for $h_0$. By Young's equality, we obtain

\[
\int_I q\left(\frac{h_0}{v}\right)h_0 = \int_I q\left(\frac{h_0}{v}\right)\frac{h_0}{v} v = \int_I \varphi\left(q\left(\frac{h_0}{v}\right)\right)v + \int_I \varphi_*\left(\frac{h_0}{v}\right)v = \int_I \varphi\left(q\left(\frac{h_0}{v}\right)\right)v + \int_I \varphi_*\left(\frac{h_0}{v}\right)v.
\]

\noindent Let $g_2 = q\left(\frac{h_0}{v}\right)$. It is a simple function with support of finite measure, so $g_2 \in (\Lambda_{\varphi,w}^0)_a$. In view of Lemma \ref{lem3}, we get

\begin{equation}\label{eq2}
P_{\varphi_*,w}(h_0) = \int_I \varphi_*\left(\frac{h_0}{v}\right)v = \int_I q\left(\frac{h_0}{v}\right)h_0 - \int_I \varphi\left(q\left(\frac{h_0}{v}\right)\right)v = \int_I g_2 h_0 - \int_I \varphi(g_2^*)w.
\end{equation}

The function $g_2h$ is integrable. So, there exists $\eta>0$ such that for any measurable subset $E \subset I$ with $mE < \eta$, we have $\int_E |g_2h| < \frac{\delta}{2}$. We will now show that for $\delta>0$, there exist $n \in \mathbb{N}$ and $E \subset I$ such that $mE <\eta$,
\begin{equation}\label{eq3}
\int_0^{mE} \varphi (g_1^*) w < \frac{\delta}{4}, \ \ \int_E |g_2 h| < \frac{\delta}{2}, \ \  \text{and} \ \ \rho_{\varphi,w}(g_1 \chi_{[n, \gamma)}) = \int_I\varphi((g_1 \chi_{[n, \gamma)})^*)w < \frac{\delta}{8}.
\end{equation}

\noindent Let $E_n=\{g_1^*> n\}= [0, t_n)$. We see that $g_1^* \chi_{E_n} \leq g_1^*$  for all $n$ and $g_1^* \chi_{E_n} \rightarrow 0$ a.e. By the Lebesgue dominated convergence theorem, for $\delta > 0$, there exists $N_1$ such that for all $n \geq N_1$,

\[
\int_I \varphi(g_1^*\chi_{E_n})w = \int_0^{mE_n} \varphi(g_1^*)w<  \frac{\delta}{4}.
\]

\noindent Since $g_1$ and $g_1^*$ are equimeasurable, we have $m\{|g_1|>n\} = m\{g_1^* > n\} = mE_n$ for all $n$. Choose $n > N_1$ such that $mE_n < \eta$, $\supp{h_0} \cap [n, \gamma) = \emptyset$, and $\rho_{\varphi,w}(g_1 \chi_{[n, \gamma)}) = \int_I\varphi((g_1 \chi_{[n, \gamma)})^*)w < \frac{\delta}{8}$. Finally, by letting $\{|g_1|>n\} = E$ for such $n$, we obtain $n \in \mathbb{N}$ and a measurable subset $E \subset I$ satisfying (\ref{eq3}). Note that $\supp{h_0} \subset [0, n)$.

Now, we define
\[\bar{u}(t) =
\begin{cases}
g_2(t), & t \in A_1 = \supp{h_0} \setminus E\\
g_1(t),& t \in A_2 = E \cup [n, \gamma) \\
0, & \text{Otherwise}.

\end{cases}
\]

\noindent The function $g_1$ is bounded on the set $A_2^c$. Moreover, $A_2^c$ is a subset of $[0,n)$. So, $g_1\chi_{A_2^c} \in (\Lambda_{\varphi,w}^0)_a$, and this implies that $S(g_1) = S(g_1 \chi_{A_2})$. Since $g_2$ is a simple function with support of finite measure, $S(g_2\chi_{A_1}) = 0$. By orthogonal subadditivity of $\rho_{\varphi,w}$, we get

\[
\rho_{\varphi,w}(\bar{u}) \leq \rho_{\varphi,w}(g_2 \chi_{A_1}) + \rho_{\varphi,w}(g_1 \chi_{A_2}) \leq \rho_{\varphi,w}(g_2 \chi_{A_1}) + \rho_{\varphi,w}(g_1 \chi_E) + \rho_{\varphi,w}(g_1\chi_{[n, \gamma)}),
\]

\noindent and by (\ref{eq3}), we have

\[
\rho_{\varphi,w}(\bar{u}) < \rho_{\varphi,w}(g_2 \chi_{A_1}) + \rho_{\varphi,w}(g_1 \chi_E) + \frac{\delta}{8}.
\]

\noindent Hence, we see that

\begin{equation}\label{eq6}
\int_I \bar{u}h + S(\bar{u}) - \rho_{\varphi,w}(\bar{u}) \geq \int_{A_1} g_2 h + \int_{A_2} g_1 h +  S(g_1) - \rho_{\varphi, w} (g_2\chi_{A_1}) - \rho_{\varphi,w}(g_1 \chi_E) - \frac{\delta}{8}.
\end{equation}

\noindent Since $g_2 \ge 0$ and $h \ge h_0 \ge 0$, we have

\[
\int_{A_1} g_2 h\ge \int_{A_1} g_2 h_0 = \int_{I \setminus E} g_2 h_0.
\]

\noindent Also, in view of (\ref{eq1}) and (\ref{eq3}), we see that

\[
\int_{A_2} |g_1 h| < \int_I |g_1 h| < \frac{\delta}{8} \ \ \ \text{and} \ \ \  \int_E g_2h_0 \leq \int_E g_2h < \frac{\delta}{2}.
\]

\noindent Then, the inequality (\ref{eq6}) becomes

\[
\int_I h\bar{u} + S(\bar{u}) - \rho_{\varphi,w}(\bar{u}) \ge
\int_{I \setminus E} g_2 h_0 -\frac{\delta}{4} +  S(g_1) - \rho_{\varphi, w} (g_2\chi_{A_1}) - \rho_{\varphi,w}(g_1\chi_{E}).
\]

\noindent Hence, we obtain
\begin{eqnarray*}
\int_I \bar{u}h + S(\bar{u}) - \rho_{\varphi,w}(\bar{u}) &\geq& \int_{I \setminus E} g_2 h_0 -\frac{\delta}{4} +  S(g_1) - \rho_{\varphi, w} (g_2\chi_{A_1}) - \rho_{\varphi,w}(g_1\chi_{E}) \\
&\geq& \int_I g_2 h_0 - \int_E g_2 h_0 - \frac{\delta}{4} + S(g_1) - \rho_{\varphi, w} (g_2) - \rho_{\varphi,w}(g_1 \chi_{E})\\
&\geq& \int_I g_2 h_0 - \int_E g_2 h_0  - \frac{\delta}{4}+ S(g_1) - \rho_{\varphi, w} (g_2) - \frac{\delta}{4}\,\,\, \text{by (\ref{eq3})}\\
&=& P_{\varphi_*,w}(h_0) -\int_E g_2h_0 +  S(g_1) - \frac{\delta}{2} \,\,\, \text{by (\ref{eq2})}\\
&\geq& P_{\varphi_*,w}(h_0)- \frac{\delta}{2} + S(g_1) - \frac{\delta}{2} \\
&>& 1+ \delta - \delta = 1. \,\,\, \text{by (\ref{eq1})}
\end{eqnarray*}

\noindent Finally, this implies that
\[
1 = \|F\| \geq F\left(\frac{\bar{u}}{\|\bar{u}\|^0}\right) = \frac{H(\bar{u}) + S(\bar{u})}{\|\bar{u}\|^0} = \frac{\int_I \bar{u}h+ S(\bar{u})}{\|\bar{u}\|^0}> \frac{1 + \rho_{\varphi,w}(\bar{u})}{\|\bar{u}\|^0} > 1,
\]
\noindent which leads to a contradiction.

\end{proof}

Next result is  the sequence analogue of the formula for the norm of a bounded linear functional on  $\lambda^0_{\varphi,w}$.

\begin{Theorem}\label{Orliczseq}
If $\varphi$ is an Orlicz $N$-function and $F$ is a bounded linear functional on $\lambda^0_{\varphi,w}$  then $F = H+S$, where $H(x) = \sum_{i=1}^{\infty} x(i) y(i)$, $ \|H\| =  \| y\|_{\mathfrak{m}_{\varphi_*,w}}$, $S$ is a singular functional vanishing on $(\lambda_{\varphi,w})_a$ and $\|F\| = \inf\{\eta>0 : p_{\varphi_*,w}(\frac{h}{\eta}) + \frac{1}{\eta}\|S\| \leq 1\}$.
\end{Theorem}

Contrary to Corollary \ref{cor:luxideal} about $M$-ideals in the Orlicz-Lorentz spaces equipped with the Luxemburg norm,  we conclude this paper by showing that $(\Lambda_{\varphi,w}^0)_a$ and $(\lambda_{\varphi,w}^0)_a$ are not $M$-ideals in $\Lambda_{\varphi,w}^0$ and $\lambda_{\varphi,w}^0$ respectively, if the Orlicz $N$-function $\varphi$ does not satisfy the appropriate $\Delta_2$ condition.

\begin{Corollary}
Let $\varphi$ be an Orlicz $N$-function which does not satisfy the appropriate $\Delta_2$ condition. Then the  order-continuous subspaces $(\Lambda_{\varphi,w}^0)_a$ or $(\lambda_{\varphi,w}^0)_a$ are not $M$-ideals in $\Lambda_{\varphi,w}^0$ or $\lambda_{\varphi,w}^0$, respectively.
\end{Corollary}

\begin{proof}

We give a proof only in the case of function spaces. 
Let $\varphi$ be an Orlicz $N$-function, which does not satisfy the appropriate $\Delta_2$ condition. Then $(\Lambda_{\varphi,w}^0)_a$ is a proper subspace of $\Lambda_{\varphi,w}^0$, and in view of Theorem \ref{Orlicz} there exists $S \in (\Lambda_{\varphi,w}^0)^*$ such that $S \neq 0$. So, choose $S \in (\Lambda_{\varphi,w}^0)^*$ such that $0 < \|S\| < 1$.  We show that there exist $u>0$ and $0 < t_0 < \gamma$ such that $h = uw\chi_{(0,t_0)}$  and $\|h\|_{\mathcal{M}_{\varphi_*,w}} + \|S\| =1$. Indeed choose $u$ satisfying $\varphi_*(u) > 1/W(\gamma)$, where $1/W(\infty) = 0$. Then  $\frac{1}{\varphi_*(u/(1-\|S\|))} < W(\gamma)$. Since $W$ is continuous on $(0, \gamma)$, there exists $0 < t_0  < \gamma$ such that  $W(t_0) = \frac{1}{\varphi_*(u/(1-\|S\|))}$. Let $h = uw\chi_{(0,t_0)}$ for such $u$ and $t_0$.
Clearly $h$ is a decreasing function. Furthermore, the interval $(0, t_0)$ is its maximal level interval since $R(0, t) = \frac{uW(t)}{W(t)} =  \frac{uW(t_0)}{W(t_0)} = R(0,t_0) = u$ for all $0 < t < t_0$, and $R(0,t_0) < R(0,t)$ for $\gamma> t> t_0$ . Hence $\frac{h^0}{w} =  u \chi_{(0, t_0)}$, and so $P_{\varphi_*,w}(h) = \int_I \varphi_*\left(\frac{h^0}{w}\right)w = \varphi_*(u)W(t_0)$. It follows that

\begin{eqnarray*}
\|h\|_{\mathcal{M}_{\varphi_*,w}} &=& \inf\left\{\epsilon > 0 : P_{\varphi_*,w}\left(\frac{h}{\epsilon}\right) \leq 1\right\} =  \inf\left\{\epsilon > 0 : \varphi_* \left(\frac{u}{\epsilon}\right) \leq \frac{1}{W(t_0)}\right\}\\
&=& \inf\left\{\epsilon > 0 : \varphi_*\left(\frac{u}{\epsilon}\right) \leq \varphi_*\left(\frac{u}{1-\|S\|}\right)\right\}
= \inf\{\epsilon > 0 : \epsilon \geq 1 - \|S\|\} = 1 - \|S\|.
\end{eqnarray*}
Thus, we have $\|h\|_{\mathcal{M}_{\varphi_*,w}} + \|S\| = 1$, which implies that $P_{\varphi_*,w} (\frac{h}{1 - \|S\|}) \leq 1$.  Now since $\varphi$ is $N$-function, $\varphi_*$ is also an $N$-function, and so $\varphi_*$ is not identical to a linear function $ku$ for any $k>0$. Hence for all $u>0$,  $\lambda> 1$ we have $\varphi_*(\lambda u) > \lambda \varphi_*(u)$. Therefore by $\frac{1}{1-\|S\|} > 1$,
\[
1 \geq P_{\varphi_*,w}\left(\frac{h}{1-\|S\|}\right) = \varphi_*\left(\frac{u}{1-\|S\|}\right)W(t_0) > \frac{1}{1 - \|S\|} P_{\varphi_*,w}(h),
\]
 which shows that
 \begin{equation}\label{eq:00}
 P_{\varphi_*,w}(h) < 1 - \|S\| = \|h\|_{\mathcal{M}_{\varphi_*,w}}.
 \end{equation}
On the other hand if we assume  that $(\Lambda_{\varphi,w}^0)_a$ is an $M$-ideal of $\Lambda_{\varphi,w}^0$ then, $1 =\|H + S\| = \|h\|_{\mathcal{M}_{\varphi_*,w}} + \|S\| \geq P_{\varphi_*,w}(h) + \|S\|$. It follows  that $P_{\varphi_*,w}(h) + \|S\| = 1$. Indeed, suppose that $P_{\varphi_*,w}(h) + \|S\| < 1$. Define $g(\lambda) = P_{\varphi_*,w}(\lambda h) + \lambda \|S\|$ for $\lambda > 0$. The function $g$ is convex, $g(0) = 0$, and $\lim_{\lambda \rightarrow \infty} g(\lambda) = \infty$. Since $g(1) = P_{\varphi_*,w}(h) + \|S\| < 1$, there exists $\frac{1}{\lambda_0} > 1$ such that $P_{\varphi_*,w}\left(\frac{h}{\lambda_0}\right) + \frac{1}{\lambda_0}\|S\| = 1$. But then, from Theorem \ref{Orlicz}, we have $1 = \|H+S\| =  \inf\{\lambda>0 : P_{\varphi_*,w}(\frac{h}{\lambda}) + \frac{1}{\lambda}\|S\| \leq 1\} > 1$, which is a contradiction.

However $P_{\varphi_*,w}(h) + \|S\| = 1$ contradicts  (\ref{eq:00}) and completes the proof.

\end{proof}

\end{document}